\documentclass[leqno,12pt]{amsart}
\usepackage{amssymb}
\usepackage{amsmath}
\usepackage{enumerate}
\usepackage{amsfonts}
\usepackage{color}
\usepackage{graphicx}
\usepackage{setspace}
\newtheorem{thm}{Theorem}[section]

\theoremstyle{remark}

\theoremstyle{remark}

\theoremstyle{definition}
\newtheorem{df}[thm]{Definition}
\newtheorem{ex}[thm]{Example}

\numberwithin{equation}{section}

\newcommand{\C}{\mathbb{C}}

\def\8{\infty}
\def\blab#1{\begin{equation}\label{#1}}
\def\elab{\end{equation}}

\def\blab#1{\begin{equation}\label{#1}}
\def\elab{\end{equation}}

\title[Group theory in pairwise comparisons: caution]{On the use of group theory to generalize elements of pairwise comparisons matrix: a cautionary note}
\author[WWK] {W.W. Koczkodaj}
\address{Computer Science\\ Laurentian University\\ Sudbury, Ontario P3E 2C6, Canada} \email{ wkoczkodaj@cs.laurentian.ca}

\author[FL]{F. Liu}
\address{School of Mathematics and Information Science \\ Guangxi University, Nanning Guangxi 530004, China}
	 \email{f\_liu@gxu.edu.cn} 
\author[VWM]{V.W. Marek}
\address{Department of Computer Science\\
	University of Kentucky\\
	304E1 Davis Marksbury Building\\
	Lexington, KY, USA}\email{marek@cs.uky.edu}

\author[JM]{J. Mazurek}
\address{Faculty of School of Business Administration in Karvina\\
	Silesian University in Opava, Czech Republic}\email{mazurek@opf.slu.cz}

\author[MM]{M. Mazurek}
\address{Faculty of Electrical and Computer Engineering\\
	Rzeszow University of Technology, Poland }\email{mirekmaz@prz.edu.pl}

\author[LM]{L. Mikhailov}
\address{School of Computer Science\\
	The University of Manchester\\
	M13 9PL, United Kingdom\\
	Manchester}\email{ludi.mikhailov@mbs.ac.uk}

\author[CO]{C. \"{O}zel}
\address{Deptartment of Mathematics,
	King Abdulaziz University, \\
	\\ 21589 Jeddah,\\
	Saudi Arabia}\email{cenap.ozel@gmail.com}

\author[WP]{W. Pedrycz}
\address{Deptartment of Electrical \& Comp Engineering,
	University of Alberta, \\
	\\ Edmonton, AB T6R 2G7,\\
	Canada}\email{wpedrycz@ualberta.ca}

\author[AP]{A. Przelaskowski}
\address{Faculty of Mathematics and Computer Science\\
	Warsaw University of Technology\\
	Koszykowa 75, 00-662 Warsaw\\
	Poland}\email{arturp@mini.pw.edu.pl}

\author[AS]{A. Schumann}
\address{Department of Cognitive Science and Mathematical Modelling, University of Information Technology and Management in Rzeszow,  Poland}\email{andrew.schumann@gmail.com}

\author[RS]{R. Smarzewski}
\address{Faculty of Cybernetics, Military University of Technology, Warsaw, Poland}\email{ryszard.smarzewski@wat.edu.pl}

\author[DS]{D. Strzalka}
\address{Faculty of Electrical and Computer Engineering\\
	Rzeszow University of Technology, Poland}\email{strzalka@prz.edu.pl}

\author[JS]{J. Szybowski}
\address{Faculty of Applied Mathematics,	AGH University of Science and Technology\\
	al. Mickiewicza 30\\ 30-059 Krakow}\email{szybowsk@agh.edu.pl}

\author[YY]{Y. Yayli}
\address{Department of mathematics\\
	Ankara University\\
	Emniyet Mahallesi, D\"{u}gol Cd. 6A\\
	06560 Yenimahalle/Ankara}\email{yayli@science.ankara.edu.tr} 

\begin{document}
	
	\begin{abstract}
		
		This paper examines the constricted use of group theory in the studies of pairwise comparisons. The presented approach is based on the application of the famous Levi Theorems of 1942 and 1943 for orderable groups.
		The theoretical foundation for multiplicative (ratio) pairwise comparisons has been provided. Counterexamples have been provided to support the theory.
		
		In our opinion, the scientific community must be made aware of the limitations of using the group theory in pairwise comparisons. Groups, which are not torsion free, cannot be used for ratios by Levi's theorems. 
		
	\end{abstract}
	\maketitle
	
	\section{Introduction}
	
	A larger collaboration has been called to raising issues about the incorrect generalization of pairwise comparisons (PC or PCs depending on the context). We also  hope that our study will also pave the way for future constructive critiques. Harmful generalizations look attractive although they are ``not even wrong'' (the phrase attributed to Wolfgang Pauli, Nobel laureate).
	
	The origins of the first documented use of pairwise comparisons have been recently traced to the thirteenth-century mathematician, polymath, philosopher, logician, Franciscan tertiary, and writer from the Kingdom of Majorca Raymond Llull  in \cite{FHHR2009}. The plural form of \textit{comparisons} is needed (as addressed in \cite{K2Y}) and has been traced to the 13th century and attributed to Ramon Llull. However, it is easy to imagine that the use of pairwise comparisons for selecting a better stone for a tool, from two stones kept in two hands, must have taken place among prehistoric humans.
	
	
	There are two types of pairwise comparisons that are commonly considered: additive and multiplicative. The additive type is based on the comparisons ``by how much'' (percentage may be used) one entity is bigger (better, more important, or similar comparisons) than another entity. The multiplicative type is a simple ratio of two entities. Compared entities may be physical objects or abstract concepts (e.g., software safety and software reliability).
	From the mathematical point of view, multiplicative (or ratio) PCs are more popular and more challenging than additive PCs. The additive PCs can be derived from the multiplicative form by a logarithmic mapping. It is not entirely clear who should be credited for such an observation but \cite{RAND} is usually assumed to be the original source. Subjective assessments (evaluations or judgments) as the primary source of data are addressed in many publications. In \cite{HK1996, KS2010, KS2016}, the logarithmic mapping is used to provide the proof of convergence and for the interpretation of the limit of convergence.
	
	Pairwise comparisons are commonly represented by a \textit{PC matrix} since matrices are one of the most universal data structures with a well established theory. In case of multiplicative PCs, it is a matrix of ratios with 1s on the main diagonal and reciprocal ($x$ and $1/x$) values in upper/lower triangles. The ratio of $B/A$ is the reciprocal value of  $A/B$. 
	
	The term \textit{ratio} was frequently used in \cite{CD2012} (even the journal title includes ratio). In fact, the multiplicative PC matrices can be regarded as a part of ratio calculus since PC matrix elements are the ratios. Rarely in pairwise comparisons, ratios are established by the actual division. When entities are abstract concepts, such as software quality and software safety, the division operation is undefined but using the ratio still makes sense (e.g., by using ``three times more important''). For this reason, their ratio is given (for example, by expert assessment). In fact, entities can represent anything we can imagine.
	
	
	Can we also assume that the ratios between entities could be elements of any algebraic group?
	The main goal of PCs is to split 1 into $n$ real values assigned to $n$ entities $E_i$, $i=1, \ldots, n$.
	
	\subsection{Contribution}
	
	
	The body of this paper is organized into four sections followed by the conclusions. Section~\ref{ratio} uses the famous Levi Theorems of 1942 and 1943 (see \cite{Levi1942, Levi1943}) to show that only torsion free abelian groups are orderable (hence disproving the claim in \cite{Wajch2019} that generalization to any algebraic group can be used). In Section~\ref{case}, an example involving a generic final exam is introduced. In Section~\ref{examples}, we present a number of counter-examples that PC matrix entries cannot be generalized to any algebraic group as stipulated in \cite{Wajch2019}.

	\section{Defending the integrity of algebraic ratio}
	\label{ratio}
	
	Matrices of ratios are called pairwise comparisons (PC) matrices.
	The ratio specification is attributed to Euclid in his famous Elements, Book V. 
	One of the oldest ratios is $\pi$, the ratio of a circle's circumference to its diameter. 
	Evidently, the ratio concept predates numbers since it was easier to express 
	that one settlement was twice as big as another rather than to count settlers in each of them. Indeed, two times bigger as a ratio relies on a smaller number than assertion: our settlement has 100 people and our enemies have only 50 people in their settlement. 
	It is also important to notice that the ratio concept 
	has not evolved since the time of Euclid, giving basis for the definition of rational, irrational and real numbers. \\
	
	In this section, we provide arguments against the use of negative ratios for pairwise comparisons. By doing this, we invalidate the unrestricted use of algebraic groups since all positive and negative numbers (excluding zero) create a group under multiplication operation. It should be observed, the question ``why ratio cannot be negative?'' is not trivial to answer. One of the most compelling explanations can be provided by using the concept of \textit{torsion}  in group theory.
	
	The first argument is a simple observation that there is no interpretation of negative ratios. For example, the statement that $x$ is -1 times larger than $y$ has no interpretation.
	The second argument is based on certain group properties investigated by Levi in \cite{Levi1942, Levi1943}) labeled in the followup research as torsion free groups. 
	
	\begin{df}
		A group $\mathcal G$ is a set $ G $ together with a (binary) group operation $ \odot $ such that:
		\begin{itemize}
			\item for all $ A,B\in G $ implies $ A\odot B \in G $,
			\item for all $ A,B,C \in G $ is $ (A\odot B) \odot C = A\odot (B \odot C) $,
			\item there exist an element $ e $ such that $ A \odot e = e \odot A$, $ \forall A\in G $,
			\item for all $ A \in G $, there exists $ A^{(-1)} \in G $ such that
			$ A^{(-1)} \odot A = A \odot A^{(-1)} = e $,
			\item a group $\mathcal G$ is called abelian if for all $ A,B\in G $ it holds that $ A\odot B = B \odot A $.
		\end{itemize}
	\end{df}
	
	The commutative property is the requirement for abelian groups but it is not necessary for a non-abelian group.
	
	In our case, the set $ G $ is composed of elements that are pairwise comparisons (ratios), and the group operation $ \odot $ is usual multiplication.
	
	All elements of $ G $, that is pairwise comparisons, should be comparable. Therefore, for any two elements $ A $ and $ B $ of the set $ G $ it holds that either $ A \geq B $ or $ A \leq B$.  
	
	Formally, this is a property of a linearly ordered group.
	
	\begin{df}
		A linearly ordered abelian group $(\mathcal G,\leq)$ is a group equipped with a relation $ \leq $ such that $\forall A, B, C \in G  $: $ A \leq B $ implies $ A\odot C \leq B \odot C $.
	\end{df}
	
	Such triad is easy to demonstrate that the set of negative pairwise comparisons (ratios) with multiplication does not satisfy the property of a linearly ordered group, since, for instance, let $A= -2$ and $B=-1$. Then $ A < B $, but by multiplication by $ C = (-1) $ we get $ AC > BC $.
	
	The problem is that the set of negative numbers contains elements of finite order. Let us recall that an element $A \in G$ is of finite order if there exists $m \in N $ such that $a^m = 1$. For instance, the $-1$ is of the order 2: $(-1)^2 = 1$. Another example of an element of finite order is an imaginary number $i$ such that $i^2 = -1$. Indeed, $i^m = 1$ for $m = 0, 4, 8, 12, \dots$. So, the $i$ is of the order 4.
	
	A group that contains no element of finite order is called a \emph{torsion free group}. Levi theorems in \cite{Levi1942, Levi1943} provide the relationship between torsion free groups and linearly ordered abelian groups.
	
	\newpage 
	
	\begin{thm}
		An abelian group $\mathcal G$ is linearly ordered if and only if it is torsion free.
	\end{thm} 
	\begin{proof}
		see \cite{Levi1942, Levi1943}.
	\end{proof}
	
	A multiplicative group of all real numbers without 0 is not torsion free since $(-1)^2=1$. To the same extent, a multiplicative group of all complex numbers without 0 is not torsion free since $i^m = 1$ for $m= 0, 4, 8, 12, \ldots$.
	Evidently, a multiplicative group of all positive real numbers is torsion free. For this reason, algebraic ratios cannot have negative values. Only positive real numbers may be used for multiplicative pairwise comparisons. For these reasons, all algebraic ratios cannot have negative values. 
	
	Finally, the famous Levi Theorem in \cite{Levi1942} explicitly states that an abelian group admits a linear order if it is torsion-free (torsion-free term was coined in 1950s but the idea if it comes from Levi). Negative ratios are popular in economy, physics, medical testing (see \cite{TFE1975}), and other branches of science but these ratios represent different concepts than the mathematical ratio used in the multiplicative variant of pairwise comparisons. 
	
	
	\section{An easy case for demonstrating how pairwise comparisons can be used}
	\label{case}
	Probably the easiest and the most compelling case for using pairwise comparisons in academia is an application to grading final exams.
	For simplicity, let us assume that we have four problems to solve; $ A, B, C, $ and $ D $.
	Evidently, hardly ever all problems are of equal level of difficulty.
	In such case, it is fair to compare $ A $ to $B$, $ A $ to $ C $, $ A $ to $ D $, $ B $ to $ C $, $ B $ to $ D $, and $ C $ to $ D $. We assume the reciprocity of PC matrix $M$: $m_{ji}=1/m_{ij}$ which is reasonable (when comparing $ B $ to $ A $, we expect to get the inverse of $ A $ to $ B $). The exam is hence represented by the following PC matrix $M$:
	
	\begin{equation} 
	M=[m_{ij}]=\begin{bmatrix}
	1 & A/B &  A/C & A/D\\
	B/A & 1 &  B/C & B/D\\
	C/A & C/B &  1 & C/D\\
	D/A & D/B &  D/C & 1
	\end{bmatrix}
	\end{equation}
	
	\noindent As previously stated, $A/B$ reads  ``the ratio between A and B'' and not as a result of the division (in case of exam problems, the division operation makes no sense to use). \\
	
	Ratios of three entities in a cycle create a triad $(A/B, A/C,B/C)$, which is said to be \textit{consistent} providing $A/B*B/C=A/C$. It is illustrated in Fig.~\ref{fig:iicycle}. Circles in Fig.~\ref{fig:iicycle} contain random numbers of dots. [A/B] reflects the assessed ratio of dots. A large enough number of dots (reflecting the \textit{numerocity}) results in the property that it is impossible to count them in a short period of time. Therefore, we need to rely on the expert opinion, hence the use of pairwise comparisons is useful. The lack of acceptable axiomatization for inconsistency (proposed in \cite{KU2018}) is one of challenges for pairwise comparisons. The convergence of inconsistency algorithms (examined in \cite{HK1996}) is another problem.
	
	
	Symbolically, in a PC matrix $M$, each triad (or a cycle) is defined by $(m_{ik},m_{ij},m_{kj})$. It is consistent if and only if $(m_{ik}*m_{kj}=m_{ij})$. When all triads are consistent (known as the \textit{consistency condition} or \textit{transitivity condition}), the entire PC matrix is considered \textit{consistent}.
	
	\begin{figure}[h]
		\centering
		\includegraphics[width=0.8\linewidth]{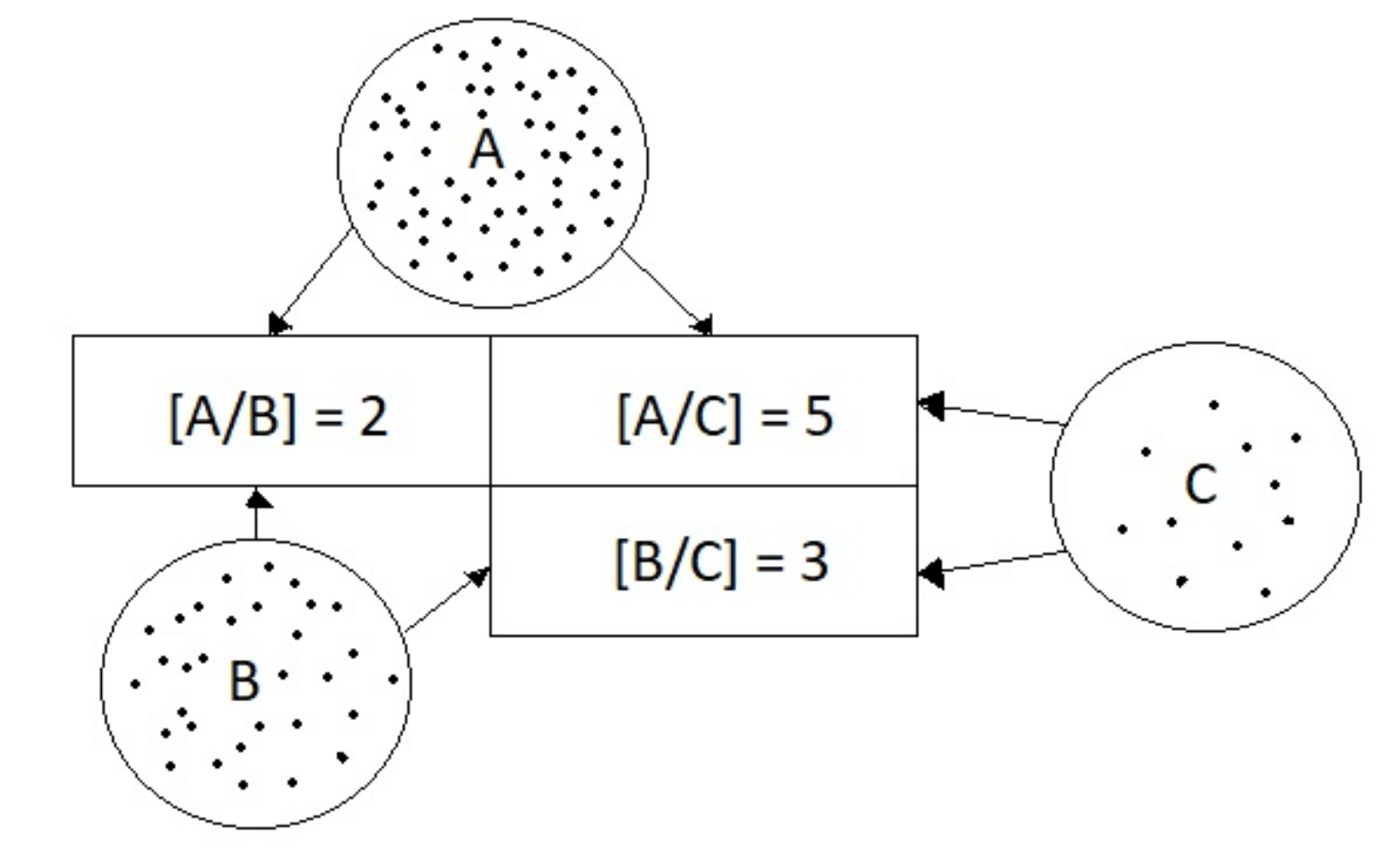}
		\caption[Fig0]{An inconsistency indicator cycle}
		\label{fig:iicycle}
	\end{figure}
	
	Looking at the above exam grading case, we have discovered a pairwise comparisons method which can be used
	to construct a PC matrix. The \textit{solution} to the PC matrix
	is a vector of weights which are geometric means of rows.
	The justification for the use of the vector $w=[v_i]$ of geometric means (GM) of rows
	is that they reconstruct the PC matrix (say $M$): $M=[v_i/v_j]$.
	
	In our example, the weights computed in this fashion (say: $ [30, 20, 10, 40]$), are now being used. By looking at the example results, we can conclude that problem $ D $ is the most difficult with the weight 40. The easiest problem is $ C $ giving one of the pairwise comparisons $ D/C = 4 $.
	
	One of the challenges of pairwise comparisons is inconsistency of assessments. It is well demonstrated by Fig.~\ref{fig:iicycle}.
	It seems that a trivial mistake took place: 6 should be in place of 5 since 2*3 gives this value. However, it unreasonably assumes that 2 and 3 are accurate assessments. We simply do not know which of the three assessments are, or are not, accurate.

	\section{Counter-examples for PC matrix element based on a non-abelian group}
	\label{examples}
	
	In this section, we consider only consistent PC matrices admitting
	elements having the negative or complex number values. According to \cite{Wajch2019}, such values are permitted since all non-zero real numbers form a group under multiplication. Similarly, all non-zero complex numbers form a multiplicative group.
	
	Our counter-examples demonstrate a problem in \cite{Wajch2019} (page 2, line 7):
	\begin{quote}
		\textit{To investigate consistency in PC, it appeared useful to consider a general case when an $n \times n$ matrix $C=[c_{i,j}]$ has all entries $c_{i,j}$ in a group $(G, \odot)$. Then $C$ is called $ \odot$-consistent if $c_{i,k}\odot c_{k,j}=c_{i,j}$ for all $i, j, k \in \{1, 2, \ldots, n\}$(see ...}
	\end{quote}
	
	\noindent The above assumption is of fundamental importance since the remaining part of \cite{Wajch2019} is based on it. 
	
	The theory in Section~\ref{ratio} and our counter-examples, provided in this section, indicate that PC matrix cannot be generalized to an arbitrary algebraic group without running into non-trivial mathematical problems such as negative ratios or complex number ratios. Evidently, there cannot be semantics for such ratios.
	
	The set $ \mathbb{R^+} $ forms a multiplicative group. For this reason, the following example of PC matrix is allowed under the fundamental assumption made in \cite{Wajch2019}.
	
	\begin{ex} \label{e1}
		Let us assume that the following consistent PC matrix $M_1$ has been created for three entities $A$, $B$, and $C$:
		\begin{equation*}
		M_1=\begin{bmatrix}
		\phantom{-}1 &-1 & \phantom{-}1\\
		-1 & \phantom{-}1& -1\\
		\phantom{-}1 &-1&  \phantom{-}1
		\end{bmatrix}
		\end{equation*}
		Such matrix is of fundamental importance for investigating ratios over arbitrary multiplicative groups since it extends the ratio to negative real numbers. Fig.~\ref{fig:fig1} attempts to ``illustrate'' PC matrix $M_1$ but it is impossible to do so without knowing semantics. 
		
		\begin{figure}[h]
			\centering
			\includegraphics[width=0.8\linewidth]{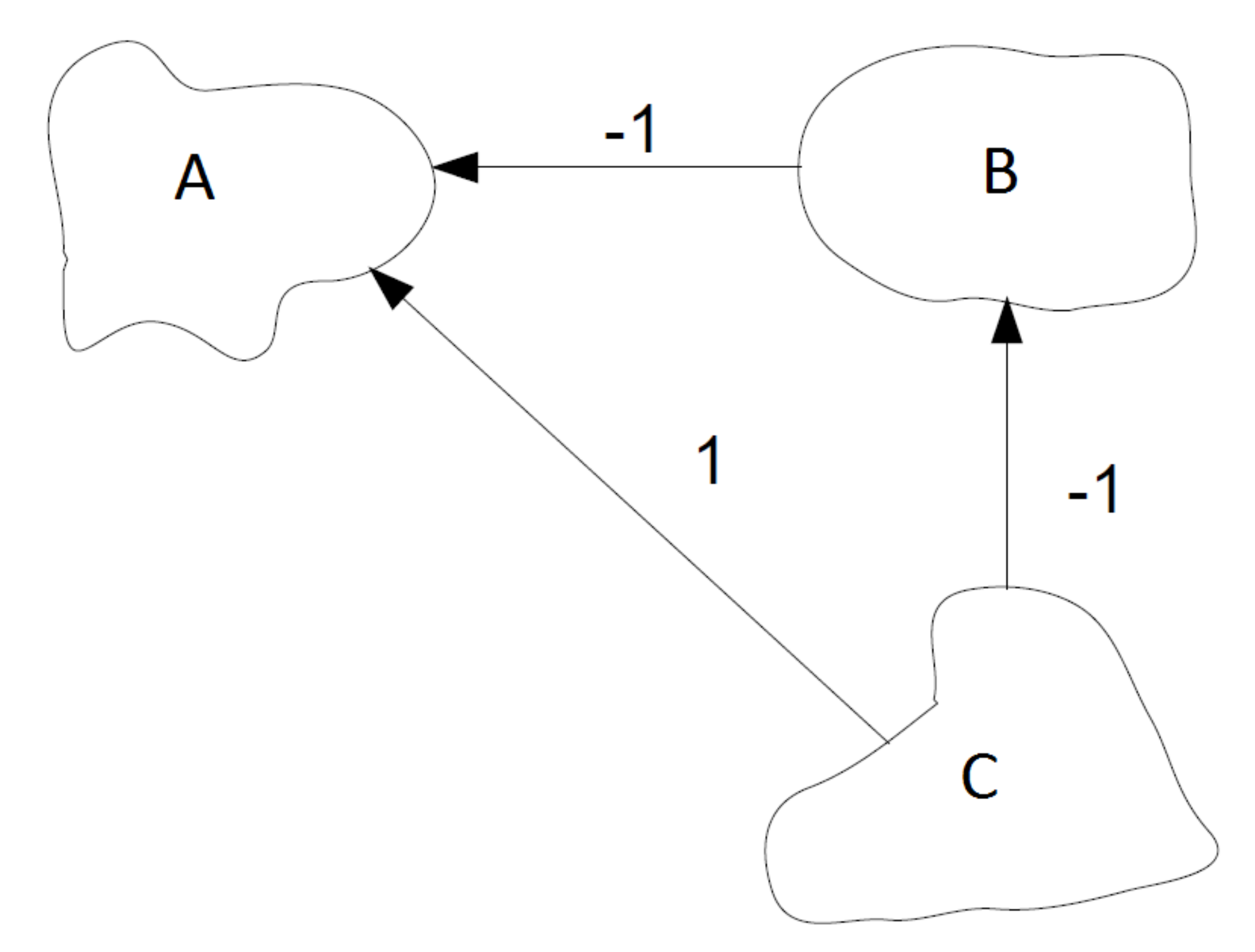}
			\caption[Three objects]{Pairwise comparisons of three entities (PC matrix $M_1$)}
			\label{fig:fig1}
		\end{figure}
		
		In PC matrix $M_1$, there are two triads of the form: $$T=(-1, 1,-1)$$
		One fails to come up with any example of three entities (physical objects or concepts) where three pairwise comparisons of these entities would not only give the above values $(-1,1,-1)$ but would be regarded as consistent pairwise comparisons.
		
		For the sake of discussion, let us consider $A/B=2$. Such expression appears to mean superiority of $A$ over $B$. $A/B=1$ stands for the equality hence $A/B=-1$ can only signify the inferiority of $A$ over $B$.
		Similar reasoning is applicable to $B/C=-1$ giving inferiority of $B$ when compared to $C$. In symbols, we get:
		$$ A \prec  B \prec C$$
		
		The equality is excluded since it is denoted by 1. Following our line of reasoning
		observe that $G=\{-1,1\}$ forms a group under multiplication.
		The geometric means (introduced and studied in \cite{RAND} as a solution) of $M_1$ rows create a vector $[-1, 1, -1]$ and one may only wonder about the meaning of such a vector. 
	\end{ex}

	%

	\begin{ex}
		The issue of complex, but non real ratios, needs to be considered.
		Consider the following consistent PC matrix:
		\begin{equation*}
		M_2=\begin{bmatrix}
		\phantom{-}1 & \phantom{-}i & \phantom{-}1\\
		-i           & \phantom{-}1 & -i\\
		\phantom{-}1 & \phantom{-}i &  \phantom{-}1
		\end{bmatrix}
		\end{equation*}
		
		$M_2$ is a consistent PC matrix since $-i*i=1$. There are 27 vectors of geometric means since the first and third coordinates are combinations of three cubic roots of imaginary unit
		$i$: $$\sqrt[3]{i}= \{\frac{\sqrt{3}}{2}+\frac{i}{2},-\frac{\sqrt{3}}{2}+\frac{i}{2},-i\}$$
		while the second one is an element of $\sqrt[3]{-1}=\{\frac{1}{2}+\frac{\sqrt{3}}{2}i,-1,\frac{1}{2}-\frac{\sqrt{3}}{2}i \}$. Similarly to $M_1$ in our Example~\ref{e1}, there is no established semantic for PC matrix $M_2$ let alone for its complex number solution.
	\end{ex}

	\begin{ex}
		Let
		\begin{equation*} 
		M_3=\begin{bmatrix}
		\phantom{-}1 & -1 &  -1 & -1\\
		-1 & \phantom{-}1 &  \phantom{-}1 & \phantom{-}1\\
		-1 & \phantom{-}1 &  \phantom{-}1 & \phantom{-}1\\
		-1 & \phantom{-}1 &  \phantom{-}1 & \phantom{-}1
		\end{bmatrix}
		\end{equation*}
		be the matrix with entries from $\C^*=\C\setminus\{0\}$ with complex multiplication.
		Evidently, $M_3$ is a consistent PC matrix (all elements of $M_3$ satisfy the condition: $m_{ij}*m_{jk}=m_{ik}$)  and is a reciprocal PC matrix (all elements of $M_3$ satisfy the condition: $m_{ij}=1\slash m_{ji}$) matrix.
		\par
		However, $4$-th root of $-1$ has four solutions in complex numbers $\C:$
		
		
		\begin{multline}
		\left\{v_k:=e^{\frac{\pi+2k\pi}{4}i}:k=0,1,2,3\right\}= \\
		\left\{
		\frac{\sqrt{2}}{2}+i\frac{\sqrt{2}}{2},
		-\frac{\sqrt{2}}{2}+i\frac{\sqrt{2}}{2},
		-\frac{\sqrt{2}}{2}-i\frac{\sqrt{2}}{2},
		\frac{\sqrt{2}}{2}-i\frac{\sqrt{2}}{2}
		\right\}
		\end{multline}
		
		hence one of the PC matrix solutions (the geometric means of rows) can be:
		$$w_k=\left(\prod_{j=1}^{4} m_{kj}\right)^{\frac{1}{4}}=v_k,\;k=0,1,2,3.$$
		Since all $v_k$ are different, we get a different weight $w_k$ for $k$th row despite all rows having identical products of all row elements equal to $-1.$ Hence we get a contradiction.
		\par
		An assumption that the ratio has a negative value $-1$ leads to the contradiction for PC matrices. For example, PC matrix $M_3$ has identical products of rows and the $4$-th root of the product (which is the geometric mean of the row) may have different values (selected from the set of four complex number solutions of $\sqrt[4]{-1}$).
		
		All coordinates of the 4*4*4*4 possible priority vectors (or weights for entities) are non-real, hence they are incomparable. There no "natural" linear order for complex numbers with or without 0. Thus, it is impossible to produce any sensible order of alternatives.
		
		Eigenvalues of $M_3$ are: $[4, 0, 0, 0]$. The corresponding eigenvectors (in columns) are:
		\begin{equation*}
		EV_{M_3}=\begin{bmatrix}
		-1 & \phantom{-}1 &  \phantom{-}1 & \phantom{-}1\\
		\phantom{-}1 & \phantom{-}0 &  \phantom{-}0 & \phantom{-}1\\
		\phantom{-}1 & \phantom{-}0 &  \phantom{-}1 & \phantom{-}0\\
		\phantom{-}1 & \phantom{-}1 &  \phantom{-}0 & \phantom{-}0
		\end{bmatrix}
		\end{equation*}
		
		The eigenvector method (introduced in \cite{Saaty1977}) assumes the Perron vector (the vector corresponding to the maximum eigenvalue) as the solution and the closest possible interpretation of the negative weight for our ``exam grading'' case would be the penalty for the perfect solution. The multiplication by $-1$ would cause it.
		
	\end{ex}

	Let us imagine that we compare two complex numbers $z_1$ and $z_2$. $\mathbb{C}$ cannot be expanded to an ordered field. It follows that neither $z_1 > z_2$ nor $z_1 < z_2$ holds hence it precludes complex numbers from being use as ratios.  
	
	%
	%
	%
	%
	%
	%
	%
	%
	%
	
	\begin{ex}
		Consider the multiplicative and consistent matrix $ M_{1} $. By logarithmic transformation this matrix is transformed into the additive matrix:
		$$M_{4}=
		\begin{bmatrix}
		0 & i\pi & 0 \\
		i\pi & 0 & i\pi \\
		0 & i\pi & 0
		\end{bmatrix}
		$$
		\\
		$$  A_{4} =
		\begin{bmatrix}
		0 & i\pi/2 & 0 \\
		-i\pi/2 & 0 & -i\pi/2 \\
		0 & i\pi/2 & 0
		\end{bmatrix}
		$$
		\\
		Evidently, consistent and reciprocal multiplicative matrix $ M_{4} $ was transformed into inconsistent and non-reciprocal additive matrix  $ A_{4} $. 
		Consistent and reciprocal multiplicative matrix $ M_{2} $ was transformed into inconsistent and reciprocal additive matrix  $ A_{2} $. 
		It follows that logarithmic mapping (in $\bf{C}$) does not preserve either consistency, nor reciprocity.
	\end{ex}
	
	Although compared entities may be of different kinds (for example, a circle and a square), comparing their area makes perfect sense.
	This interpretation has not been altered since antiquity
	and it should remain in this form until we find a good reason to change it.
	
	The negative ratio is not a major problem. The main problem is the contradiction to which the use of negative numbers brings us. It is illustrated by Example~\ref{e1}. The semantics for ratios as a progression: $$2 > 1 > 1/2 > -1$$
	seems reasonable but our discussion demonstrates that leads to the contradiction in our Example~\ref{e1}. In other words, our Example~\ref{e1} implies that the negative ratio cannot be used and it is illustrated by Fig.~\ref{fig:fig1}.
	
	
	The inconsistency in pairwise comparisons may occur when there are superfluous pairwise comparisons. Such phenomenon often occurs when dealing with subjective data. 
	For the usually assumed reciprocity, there are  $\frac{n(n-1)}{2}$ pairwise comparisons but the minimal number of pairwise comparisons for $n$ entities, without generating inconsistency, is $n-1$ (as observed in \cite{KS2015}).
	More than $n-1$ pairwise comparisons for $n$ entities may result in inconsistency.
	An example in \cite{K1993} compares three entities: $A$, $B$, and $C$.
	Two ratios: $x=A/B=2$ and $z=B/C=3$ imply $y=A/C=6$
	but if we additionally assess $y=A/C$ to be 5, it makes all three pairwise comparisons inconsistent as illustrated by Fig~\ref{fig:fig1}. The inconsistency indicator defined in \cite{K1993} has been referred to (e.g., in \cite{Wajch2019}) as Koczkodaj's inconsistency indicator (abbreviated to $Kii$).
	
	The natural question is "why allow superfluous comparisons?"
	The assessment of $y$ to be 5 may be more accurate than the other two assessments since it is impossible to say which assessment is (or is not) accurate.
	In such case, we cannot reduce the number of comparisons, and thus, there is a need to process all assessments.
	
	The use of arbitrary abelian groups, introduced in \cite{CavDap09} 
	also allow negative real numbers but proponents carefully assumed only strictly positive real values for elements of a group and used abelian linearly ordered (\textit{alo}) groups in \cite{CB2018}. 

	\section{Conclusions}
	
	This study provides the mathematical theory and counter-examples that the multiplicative pairwise comparisons (ratios) cannot be elements of an arbitrary group. Algebraic ratios should remain positive real numbers unless a sound mathematical theory or examples are provided for other groups.
	
	There is a high expectation of utilizing the proposed computer science theory. Without at least the utilization perspectives, it is a "what if" theory and our study shows what would take place if negative or complex numbers were used to construct the multiplicative pairwise comparison matrices. 
	
	In the philosophy of mathematics, Benjamin Peirce is known for the statement "Mathematics is the science that draws necessary conclusions."
	Based on Levi's Theorems, using the negative ratio in pairwise comparisons draw necessary conclusions of ``they cannot be use'' and it is of considerable importance for all future generalization attempt of pairwise comparisons. 
	
	In his autobiography (see \cite{Ford}), Henry Ford wrote:
	\begin{quote}
		and I remarked: "Any customer can have a car painted any colour that he wants so long as it is black."
	\end{quote}
	
	For the time being, using group elements for constructing multiplicative pairwise comparisons matrices can take place as long as the group is over  $\mathbb{R^+}$.
	
	\section*{Acknowledgment}
	
	Developing this text has been a true collective effort. In particular,  familiarity with the practice of the
	applications of pairwise comparisons and mathematics were needed to conduct this study. The authors recognize the efforts of Martha O'Kennon  and Tiffany Armstrong in proofreading this text.

\end{document}